\documentclass[12pt,psamsfonts]{amsart}
\usepackage{graphicx, psfrag}

\usepackage{amsmath, amssymb, amsfonts, amsthm, latexsym}
\usepackage[latin1]{inputenc}
\usepackage{dsfont}
\usepackage{amscd}
\usepackage[all]{xy}

\newcommand{\PP}{\mathbb P}
\newcommand{\QQ}{\mathbb Q}
\newcommand{\RR}{\mathbb R}
\newcommand{\ZZ}{\mathbb Z}

\DeclareMathOperator{\im}{im}

\DeclareMathOperator{\Proj}{Proj} 
\DeclareMathOperator{\rank}{rank}

\DeclareMathOperator{\Syz}{Syz}

\newtheorem{proposition}{Proposition}[section]
\newtheorem{theorem}[proposition]{Theorem}
\newtheorem{lemma}[proposition]{Lemma}
\newtheorem{corollary}[proposition]{Corollary}
\newtheorem{conjecture}[proposition]{Conjecture}
\newtheorem{theoremintro}{Theorem}

\theoremstyle{definition}

\newtheorem{definitionintro}[theoremintro]{Definition}
\newtheorem*{acknowledgements}{Acknowledgements}
\newtheorem{definition}[proposition]{Definition}
\newtheorem{remark}[proposition]{Remark}
\newtheorem{example}[proposition]{Example}
\newtheorem{construction}[proposition]{Construction}
\newtheorem{notation}[proposition]{Notation}

\renewcommand{\O}       {\mathcal{O}} 

\newcommand  {\shS}     {\mathcal{S}}

\newcommand  {\lra}     {\longrightarrow}
\newcommand  {\ra}      {\rightarrow}
\newcommand{\comdots}{ , \ldots , }
\newcommand{\komdots}{ , \ldots , }
\newcommand{\plusdots}{ + \ldots + }

\newcommand  {\dual}    {\vee}

\newcommand{\NN}{\mathbb{N}}
\newcommand{\ainv}[1][R]{a(#1)} 

\newcommand{\test}{{u}}
\newcommand{\dip}{{d}} 
\newcommand{\nog}{{n}} 
\newcommand{\degg}{{a}} 

\newcommand{\firstsyzbig}{\underline{\Syz_1}}

\begin{document}

\title[Generic bounds for tight closure]
{Generic bounds for Frobenius closure and tight closure}

\author[Holger Brenner and Helena Fischbacher-Weitz]{Holger Brenner and Helena Fischbacher-Weitz}
\address{Fachbereich Mathematik/ Informatik, Universit\"at Osna\-br\"uck, 49069 Osna\-br\"uck, Germany}

\email{hbrenner@uni-osnabrueck.de, hfischba@uni-osnabrueck.de.}


\begin{abstract}
We use geometric and cohomological methods to show that given a degree bound
for membership in ideals of a fixed degree type in the polynomial
ring
$P=k[x_0, \ldots, x_{\dip}]$, one
obtains a good generic degree bound for membership in the tight
closure of an ideal of that degree type in any standard-graded
$k$-algebra
$R$ of dimension $\dip+1$. This indicates that the tight closure of an ideal behaves more
uniformly than the ideal itself. Moreover, if $R$ is normal, one
obtains a generic bound for membership in the Frobenius closure.
If $\dip \leq 2$, then
the bound for ideal membership in $P$ can be
computed from the known cases of the Fr\"oberg conjecture and yields
explicit generic tight closure bounds.
\end{abstract}

\maketitle


\section*{Introduction}

Let $P=k[x_0 \comdots x_{\dip}]$ be a standard-graded polynomial ring over a field
$k$ and let $\degg_1 \comdots \degg_\nog$ be natural numbers. For a
family $f_1 \comdots f_\nog$ of homogeneous polynomials of degree
$\deg(f_i)=\degg_i$ we look at the ideal $I=(f_1 \comdots f_\nog)$.
The Fr\"oberg conjecture, which has been proved in dimension $\dip \leq
2$, claims that the Hilbert function
$$m \mapsto H(m)= \dim_k\; (P/I)_m$$
has an easy description given by the coefficients of a certain power
series defined by the degrees $\degg_1 \comdots \degg_\nog$, provided that the
$f_i$ are choosen generically.  In particular, this conjecture gives
for $\nog \geq \dip +1$ an implicitly defined degree bound $m_0$ for
ideal membership (depending only on the degrees $\degg_1 \comdots
\degg_\nog$), by which we mean that $P_{\geq m_0} \subseteq (f_1 \comdots f_\nog)$. This degree bound is the smallest number where the predicted generic Hilbert function vanishes.

Now let $R$ be any standard-graded $k$-algebra of dimension $\dip +1$. Does there exist a similar generic degree bound for ideal membership? Already the parameter case ($\nog=\dip +1$) shows that
there is no such degree bound which is independent of $R$. In
contrast, we will show that if we look at the tight closure of
the ideal instead, then there does exist a generic degree bound,
depending only on the dimension and the degrees $\degg_i$. If $m_0$ is the generic degree bound in the polynomial ring, then $m_0 + \dip$ is a generic tight closure
bound for all standard-graded $k$-algebras of dimension $\dip+1$ over a field of positive characteristic (see Theorem~\ref{maintheoremintro} below and 
Theorem~\ref{maintheorem} in the main text). This means that the containment in the tight closure behaves more
uniformily than the containment in the ideal. This is in stark contrast to the often expressed opinion that tight closure is difficult to compute and behaves mysteriously.

We briefly recall the notion of tight closure. The theory of tight closure
has been developed by Hochster and Huneke
since 1986 (see \cite{hochsterhuneketightclosure, hochsterhunekebriancon,
hunekeapplication}). It assigns to every ideal $I$ in a Noetherian
ring containing a field an ideal $I^* \supseteq I$, the tight
closure of $I$. In positive characteristic, tight closure and the related notion
of Frobenius closure are defined as follows.
\begin{definitionintro}
Let $R$ be a Noetherian ring containing a field of characteristic
$p>0$, and let $I=(f_1 \comdots f_{\nog}) \subseteq R$ be an ideal.
Let
\[ \begin{array}{ll} I^{[q]} & := (f_1^q \comdots f_\nog^q) \subseteq
R \; \text{for} \; q=p^e,\; e \in \NN \\
I^F &:= \{ x \in R: x^q \in I^{[q]}\; \text{for some}\; q=p^e \}
 \\
I^* &:=\{ x \in R: \exists \test \notin \text{min. prime}: \test x^q \in
I^{[q]}\; \text{for almost all}\; q=p^e\} \end{array} \]
$I^F$ is called the \emph{Frobenius closure} of $I$ and $I^*$ is
called the \emph{tight closure} of $I$.
\end{definitionintro}

In a regular ring, such as for example a polynomial ring over a field, every ideal is tightly closed (i.e. $I=I^*$), and it is an important feature of tight closure theory that we can often generalize statements about ideal membership in regular rings to non-regular rings if we replace the ideal by its tight closure. A typical example is the tight closure version of the Brian\c
con-Skoda theorem \cite[Theorem~5.7]{hunekeapplication}. Our results strongly support this principle. Here the generic tight closure result follows from the regular case by cohomological vanishing conditions and by semicontinuity.

The search for degree bounds for tight closure is a classical topic (\cite{smithgraded}, \cite{hunekeparameter}, \cite{brennerslope}, \cite{brennerlookingstable}, \cite{brennerlinearbound}). The first result in this direction is that for parameters ($\nog= \dip+1= \operatorname{dim} (R)$) of degrees $\degg_1 \komdots \degg_\nog$ in a graded ring $R$ the inclusion $R_{\sum \degg_i} \subseteq I^*$ holds
(see \cite[Theorem 2.9]{hunekeparameter}). This is an application of
the Brian\c con-Skoda theorem, and beside being parameters no further
genericity condition is required. For more than
$\operatorname{dim}(R)$ generators one can not expect a degree bound
without any furher assumptions. For constant degree $\degg$ one
obtains the degree bound $ \frac{(
\operatorname{dim}(R)-1)n}{\nog-1}a$ under the condition that the
top-dimensional syzygy bundle is strongly semistable (see
\cite[Corollary 2.8]{brennerlinearbound}). This result is best
possible in dimension two, but not in higher dimensions. We will show
that in dimension three, under sufficient genericity conditions, there
exists a degree bound which behaves asymptotically like 
$\frac{n + \sqrt{n}}{n-1} a$. This can be deduced from our main theorem.

\begin{theoremintro}
\label{maintheoremintro}
Let $\dip \geq 1$. Let $\degg_1, \ldots, \degg_\nog$ be natural numbers, $\nog \geq \dip +1$. Let $m \in \NN$. Suppose that there exist $g_1 \comdots g_\nog \in P=k[x_0, \ldots, x_{\dip}]$, $\deg(g_i)=\degg_i$, such that $P_m \subseteq (g_1 \comdots g_\nog)$. Let $R$
be a $(\dip +1)$-dimensional standard-graded $k$-algebra with a graded
Noether normalization $P \subseteq R$. Let $f_1 \comdots f_\nog \in R$
denote elements of degree $\deg(f_i)=\degg_i$.

\begin{enumerate}
\item[(a)] The containment
\[ R_{m+\dip} \subseteq (f_1 \comdots f_\nog)^* \]
holds in the generic point of the space parametrizing these tuples and also countably generically 
(for these notions cf. Definition \ref{genericdef}(b),(c)).
\item[(b)]
If $R$ is normal, then for generic elements  $f_1 \comdots f_\nog \in
R$ with $\deg(f_i)=\degg_i$, we have
$R_{m+\dip +1} \subseteq (f_1 \comdots f_\nog)^F$.

\item[(c)]
If $R$ is Cohen-Macaulay with $a$-invariant $\ainv$ and of dimension $\geq 2$,
then
$R_{m+\dip +1 + \ainv } \subseteq (f_1 \comdots f_\nog)$
for generic elements $f_1 \comdots f_\nog \in
R$ with $\deg(f_i)=\degg_i$.
\end{enumerate}
\end{theoremintro}

We give a rough outline of the steps in our proof (part (a)).
\begin{enumerate}
\item[1.] We choose a homogeneous Noether normalization $P \subseteq R$ and choose $g_1 \comdots g_\nog \in P$ for which $P_m \subseteq (g_1 \comdots g_\nog) \subseteq P$ holds. This containment tells us that the last module in the minimal free resolution of the ideal $(g_1 \comdots g_\nog)$ is of the form 
\[\bigoplus_{i=1}^{n_\dip} P(-\alpha_{\dip,i})\; \text{with}\;
\alpha_{\dip,i} \leq m + \dip\] (Lemma \ref{alpha-estimate-lemma}).
\item[2.] The pullback of this resolution from $P$ to $R$
(cf. (\ref{minrespullback})) is still exact on
the punctured spectrum and on $Y=\Proj R$. The last syzygy bundle 
$\Syz_\dip:=\Syz_\dip(g_1 \comdots g_\nog)$ on $Y$ of this sheaf resolution is the splitting bundle $\Syz_\dip = \oplus \O_Y(- \alpha_{\dip,i})$ (cf. (\ref{topsyzsplittingR})). For $\tilde m \geq m+\dip$, we obtain 
\[\Syz_\dip(\tilde m) \cong \bigoplus_{i=1}^{n_\dip} \O_Y(\tilde m-\alpha_{\dip,i})\; \text{with}\; \tilde m-\alpha_{\dip,i} \geq 0 \; \text{for all}\; i.\]
Note that all twists are non-negative.
\item[3.] We deduce from the above that $H^\dip(Y,\Syz_\dip(\tilde m))$ has the
following ``tight closure'' property: For every $c \in
H^\dip(Y,\Syz_\dip(\tilde m))$, there exists $\test \in R$ (not in any minimal prime) such that $\test
F^{e*}(c)=0$ for almost all $e \in \NN$. 
\item[4.] By ``cohomology hopping'' (Lemma~\ref{cohomologyhopping-tight}), 
 $H^1(Y,\Syz_1(g_1 \comdots g_\nog))$ then has the same property.
\item[5.] In order to generalize from $g_1 \comdots g_\nog$ to generic
elements $f_1 \comdots f_\nog \in R$ (where each 
$f_i$ is homogeneous of the same degree as $g_i$), we introduce a suitable parametrizing space $Z$ 
whose points represent the coefficients of the $f_i$. The syzygy
bundles $\Syz_1(f_1 \comdots f_n)$ 
for various choices of $f_1 \comdots f_n$ may be viewed as the fibers
$(\firstsyzbig)_z$ ($z \in Z$) of one ``big'' vector bundle $\firstsyzbig \lra Y
\times Z \lra Z$.
\item[6.] The semicontinuity theorem implies that the cohomological tight 
closure criterion holds not only for $H^1(Y,\Syz_1(g_1 \comdots g_\nog))$ 
(from Step~3), but also for  
$$H^1(Y,\Syz_1(f_1 \comdots f_n))=H^1((Y \times Z)_z,(\firstsyzbig)_z)$$
for sufficiently generic points $z \in Z$, i.e. for ``countably generic'' 
choices of $f_1 \comdots f_\nog \in R$, (Lemma~\ref{semicont}). 
This implies that $R_{\tilde m} \subseteq (f_1 \comdots f_\nog)^*$ for countably generic choice (Lemma~\ref{tools-tightclosure}(b)).
\end{enumerate}

The principle behind Theorem~\ref{maintheoremintro} is that every
generic ideal inclusion bound in the polynomial ring gives generic
tight closure degree bounds in any graded ring. Generic bounds in the
polynomial ring are directly related to the Fr\"oberg conjecture
(Section \ref{ideal-inclusion}). Hence known cases of this conjecture
together with our main theorem yield many new concrete degree bounds
for tight closure (Section \ref{examples}). For example, the
Fr\"oberg conjecture is known to hold true in the three-dimensional polynomial ring $k[x_0,x_1,x_2]$, and if moreover we have constant degrees 
$\degg_1 = \ldots = \degg_\nog=:\degg$ with $\nog
\geq 4$, then we can 
deduce the explicit ideal inclusion bound
\[m_0= \Bigm \lceil \frac{1}{2(\nog-1)} \bigm(
3-3\nog+2\degg \nog + \sqrt{1-2\nog + \nog^2 + 4 \degg^2 \nog}
\bigm)
\Bigm
\rceil  \]
in the polynomial ring (Lemma~\ref{dim3bound}). 
Theorem~\ref{maintheoremintro} then implies that
\[m_0+2 = \Bigm \lceil \frac{1}{2(\nog-1)} \bigm(
-1+\nog+2\degg \nog + \sqrt{1-2\nog + \nog^2 + 4 \degg^2 \nog} 
\bigm)\Bigm \rceil  \]
is a generic tight closure bound in any standard-graded $k$-algebra of 
dimension three. As indicated above, for 
$\degg \gg 0$, these bounds asymptotically 
behave like $\frac{n+\sqrt{n}}{n-1} \degg$
(Remark~\ref{asymptotic-large-degree}), which is a considerable
improvement on previously known bounds.

\begin{acknowledgements} 
This research was funded by an EPSRC first grant (``Tight closure and strong semistability of vector bundles in higher dimensions'') held at the University of Sheffield and at the Universit\"at Osnabr\"uck. 
In order to understand the behaviour of the Fr\"oberg function
and to find its zeros, it has been very helpful to have the
computer algebra packages \emph{Axiom} and \emph{CoCoA} as 
well as a \emph{Lisp} program by Thomas Fischbacher which, 
among other things, computed a conjectural formula for the 
smallest positive zero using continued fractions. 
We would like to thank Almar Kaid for his very useful implementation 
of a semistability test in \emph{CoCoA} and for his help with the 
\emph{CoCoA} computations in Example \ref{tabellen}. We also thank the referee,
Axel St\"abler and Ramesh Satkurunath for useful comments.
\end{acknowledgements}

\section{Cohomological vanishing conditions for tight closure and for Frobenius closure}
\label{prelim}
In this section, we recall the cohomological interpretation of Frobenius
closure and tight closure. The methods presented in this section
will be used later to find degree bounds for the Frobenius closure
and tight closure of a generic primary ideal in a standard-graded
ring of dimension at least two.

Let $k$ be an algebraically closed field of characteristic $p >0$, and let $R$
be a normal, $\NN$-graded, finitely generated $k$-algebra of
dimension
$\dip+1$, $\dip \geq 1$. Assume further
that the $\dip$-dimensional projective variety $Y:=\Proj R$ is
covered by the open subsets $D_+(x)$ with $x \in R_1$
(this holds in particular if $R$ is standard-graded, i.e. if
it is finitely generated by $R_1$ as an $R_0$-algebra). Let $I=(f_1
\comdots f_\nog)$ be a homogeneous $R_+$-primary ideal (i.e.
$D(I)=D(R_+)$). A homogeneous free complex which is exact on the punctured spectrum $D(R_+)$ (for example, a
resolution)
\[ \xymatrix{ \ldots \ar[r] & R^{n_3} \ar[r] &
R^{n_2}
\ar[r] & R^\nog \ar[r]^<<<<<{f_1 \comdots f_\nog} & R \ar[r] & R/I
\ar[r] & 0} \]
induces an exact complex of sheaves on $Y=\Proj R$:
\begin{equation}
\label{one}
\xymatrix{
 \ldots \ar[r]  &
\bigoplus_{i=1}^{n_2} \O_Y(-\alpha_{2,i}) \ar[r] &  \bigoplus_{i=1}^{n}
\O_Y(-\alpha_{1,i}) \ar[r]^<<<<<{f_1 \comdots f_n} & \O_Y \ar[r] & 0}
\end{equation}
where the $\alpha_{1,i}$ are the degrees of the
$f_i$ (the term $R/I$ on the right becomes~$0$).

Let $\Syz_j:= \Syz_j (f_1 \comdots f_\nog)$ be the image and kernel in the $j$-th
term, such that we have short exact sequences of sheaves for $j=1$,
\begin{equation}
\label{two}
\xymatrix{
0 \ar[r] & \Syz_1 \ar[r] &
\bigoplus_{i=1}^\nog \O_Y(-\alpha_{1,i}) \ar[r]^<<<<<{f_1 \comdots f_\nog} &
\O_Y \ar[r] & 0 \, ,}
\end{equation}
and for $j =2 \comdots \dip$,
\begin{equation}
\label{three}
\xymatrix{
0 \ar[r] & \Syz_j \ar[r] &
\bigoplus_{i=1}^{n_j} \O_Y(-\alpha_{j,i}) \ar[r] &
\Syz_{j-1} \ar[r] & 0 \, .}
\end{equation}

\newcommand{\twistell}{{\ell}}

The covering property $Y=\bigcup_{x \in R_1} D_+(x)$ implies that
$\O_Y(\twistell)$ is locally free for all $\twistell \in \ZZ$ (see the proof of
\cite[II.5.12]{hartshornealgebraic}); it follows
inductively that the sheaves $\Syz_j$ are locally free, since they
are kernels of surjective homomorphisms between locally free
sheaves. The first syzygy bundle $\Syz_1$ and the ``top-dimensional'' syzygy
bundle $\Syz_\dip$ are most important for us.

Now let $f \in R_m$, which can be identified with  $H^0(Y, \O_Y(m))$
since $R$ is normal and of dimension $\geq 2$. For $j \in \{1, \ldots, \dip\}$, we inductively define cohomology classes $c_j \in H^j(Y, \Syz_j(m))$ as follows.
For $j=1$, consider the long exact cohomology sequence associated to the $m$-twist
of (\ref{two}), let $\delta^0$ denote the connecting homomorphism $H^0(Y, \O_Y(m)) \rightarrow H^1(Y, \Syz_1(m))$, and set $c_1:=\delta^0(f) \in H^1(Y, \Syz_1(m))$. For $j=2 \comdots \dip$, consider the cohomology sequence associated to the $m$-twist of (\ref{three}), let $\delta^{j-1}$ denote the connecting homomorphism $H^{j-1}(\Syz_{j-1}(m)) \rightarrow H^j(\Syz_j(m))$, and set $c_j:=\delta^{j-1}(c_{j-1}) \in H^j(Y, \Syz_j(m))$.

\begin{lemma}[``Cohomology hopping'' for Cohen-Macaulay rings] \label{CM-cohomologyhopping}
If $R$ is Cohen-Macaulay of dimension $\geq 2$, then $f \in I$
if and only if $c_j=0$ for some (or all) $j \in \{1 \comdots \dip\}$.
\end{lemma}

\begin{proof} 
If $R$ is Cohen-Macaulay, then $H^j(Y, \O_Y(\twistell))=0$ for
$j=1 \comdots \dip-1$ and all $\twistell \in \ZZ$. Hence for $j=2 \comdots \dip-1$, part of the long exact cohomology sequence associated to the $m$-twist of (\ref{three}) looks like
\begin{equation} \label{CM-coh-isomorphisms} 
0 \longrightarrow H^{j-1}(\Syz_{j-1}(m)) \stackrel{\delta^{j-1}}{\longrightarrow} H^j(\Syz_j(m)) \longrightarrow 0, \end{equation}
so $\delta^j$ is an isomorphism for these $j$. 
Moreover, $\delta^{\dip-1}$ is injective. Hence we have
\[ f \in I \Leftrightarrow c_1=\delta^0(f)=0 \Leftrightarrow c_2=\delta^1(c_1)=0 \Leftrightarrow \ldots \Leftrightarrow c_{\dip}=\delta^{\dip-1}(c_{\dip-1})=0.\]
\end{proof}

Even without the Cohen-Macaulay assumption, the containment in the Frobenius
closure and in the tight closure can be expressed in terms of the
cohomology classes $c_j$.

Let $F: Y \ra Y$ denote the absolute Frobenius morphism.
If $c$ is a class in $H^j(Y, \mathcal{S})$ for some locally free sheaf
$\mathcal{S}$ and some $j >0$, and if $q=p^e$ with $e \in \NN$, then
we write $c^q$ for the image of $c$ under the Frobenius pullback
${F^e}^*: H^j(Y, \mathcal{S}) \ra H^j(Y, {F^e}^* \mathcal{S})$.
The next two rather technical lemmas describe two cohomological vanishing conditions for syzygy
bundles which imply tight closure and Frobenius closure containments.

\begin{lemma}[Cohomology hopping and a cohomological criterion 
for tight closure]
\label{tools-tightclosure}
\label{cohomologyhopping-tight}
Let $\twistell$ be such that
$\twistell \geq 0$ and $H^j(Y,\O_Y(\twistell'))=0$ for $j=1
\comdots \dip$ and for all $\twistell' \geq \twistell$ (such $\twistell$ always exists by Serre vanishing \cite[Theorem III.5.2]{hartshornealgebraic}). 
Let $m \geq\alpha_{j,i}$ for $j=1 \comdots \dip$ and $i=1 \comdots n_j$,
where $\alpha_{j,i}$ are the twists appearing in (\ref{one}). 
\begin{enumerate}
\item[(a)]
Let $e \in \NN, q=p^e$, and let $\twistell' \geq \twistell$.
Then we have a surjective map
\[ H^0(Y, \O_Y(qm+\twistell')) \stackrel{\delta^0}{\lra} H^1(Y,F^{e*}(\Syz_1(m)) \otimes \O_Y(\twistell')), \]
isomorphisms
\[ H^{j-1}(Y,F^{e*}(\Syz_{j-1}(m)) \otimes \O_Y(\twistell')) \stackrel{\delta^{j-1}}{ \lra} H^j(F^{e*}(\Syz_j(m)) \otimes \O_Y(\twistell'))\]
for $j=2 \comdots \dip-1$, and an injective map
\[ H^{\dip-1}(Y,F^{e*}(\Syz_{\dip-1}(m)) \otimes \O_Y(\twistell')) \stackrel{\delta^{\dip-1}}{\lra} H^\dip(F^{e*}(\Syz_\dip(m)) \otimes \O_Y(\twistell')).\]
\item[(b)] Let $e \in \NN, q=p^e$. Let $f \in R_m$, and let
$c_j \in H^j(Y,\Syz_j (m))$ be the cohomology 
classes defined above. Let $\test \in R_{\twistell'}$,
$\twistell' \geq \twistell$,
$\test \notin$ minimal prime. Then 
$\test f^q \in I^{[q]}$ if and only if $\test c_j^q=0$ for some (or all) 
$j \in \{1 \comdots \dip\}$.
\item[(c)] Suppose there exists $j \in \{1 \comdots \dip\}$ and 
$\twistell_0 \geq \twistell$ such that
$H^j(Y, {F^e}^*(\Syz_j(m)) \otimes \O_Y(\twistell')) =0$ for all $e \gg 0$ and all $\twistell' \geq \twistell_0$. Then the same cohomology property holds also for $j' \leq j$ and
$R_m \subseteq I^*$.
\end{enumerate}
\end{lemma}
\begin{proof}
We start with (\ref{two}) and (\ref{three}), take the $m$-twist,
then the $e$-th
Frobenius pullback, then the $\twistell'$-twist
(the sequences remain exact because the sheaves are locally free) and
finally we take cohomology.
From (\ref{two}) we thus obtain
\begin{equation} \label{cohomologysequencefromtwo} 
\xymatrix{ \ldots  \ar[r] & \bigoplus_i H^{0}\Big(Y,
\O_Y(q(m-\alpha_{1,i})+\twistell')\Big) \ar[r]^<<<<<{f_1^q \comdots f_\nog^q} &
H^{0} \Big(Y, \O_Y(qm+\twistell') \Big)}\end{equation}
\[ \xymatrix{ \ar[r]^<<<<<{\delta^0} & H^1\Big(Y,
\big({F^e}^*(\Syz_j(m))\big) \big(\twistell'\big) \Big) \ar[r] & \bigoplus_i H^1\Big(Y,
\O_Y(q(m-\alpha_{1,i})+\twistell')\Big) \ar[r] & \ldots}\]
and for $j=2 \comdots \dip$ we get from (\ref{three}),
\begin{equation}\label{cohomologysequencefromthree} 
\xymatrix{ \ldots  \ar[r] & \bigoplus_i H^{j-1}\Big(Y,
\O_Y(q(m-\alpha_{j,i})+\twistell')\Big) \ar[r] &
H^{j-1}\Big(Y, \big({F^e}^*(\Syz_{j-1}(m))\big)\big(\twistell'\big) \Big) }\end{equation}
\[\xymatrix{\ar[r]^<<<<<{\delta^{j-1}} &
H^j \Big(Y,\big({F^e}^*(\Syz_j(m))\big) \big(\twistell'\big) \Big)
\ar[r]& \bigoplus_i H^{j}\Big(Y, \O_Y(q(m-\alpha_{j,i})+\twistell')\Big)
\ar[r] & \ldots} \]
By the assumptions made on $\twistell'$ and $m$, we have
$$H^j(\O_Y(q(m-\alpha_{j,i})+\twistell')=0$$
for $j=1 \comdots \dip$, so
$\delta^{j-1}$ is surjective for
$j=1$, bijective for $j=2 \comdots \dip-1$ and injective for $j=\dip$,
which proves part (a). 

With $\test, f, c_j$ as in part (b), it follows
that
\[ \test f^q \in I^{[q]}=\im(f_1^q \comdots f_n^q) 
\Leftrightarrow \test c_1^q=\delta^0(\test f^q)=0 \Leftrightarrow \test c_2^q=\delta^1(\test c_1^q)=0 
\Leftrightarrow \ldots\]
\[ \ldots \Leftrightarrow \test c_{\dip}^q=\delta^{\dip-1}(\test c_{\dip-1}^q)=0,\]
which proves part (b).

For part (c), note that it is always possible to find 
$\twistell' \geq \twistell_0$ and $\test \in R_{\twistell'}$,  
$\test \notin$ minimal prime. Now let $f \in R_m$ and let $c_j$ be the associated cohomology class in $H^j(\Syz_j(m))$. 
If $H^{j}(Y, {F^e}^*(\Syz_j (m)) \otimes
\O_Y(\twistell')) =0$ for $e \gg 0$, then we have $\test c_j^q=0$ 
for almost all $q=p^e$. By part (b), this
implies that $\test f^q \in I^{[q]}$ for almost all $q$, and hence
$f \in I^*$. Since $f \in R_m$ was chosen arbitrarily, this proves part (c).
\end{proof}

\begin{lemma}[A cohomological criterion for Frobenius
closure]
\label{tools-Frobeniusclosure}
Let $m > \alpha_{j,i}$ for all $j,i$,
where $\alpha_{j,i}$ are the Betti twists
from (\ref{one}). Fix $j_0 \in \{1
\comdots \dip \}$. If $H^{j_0} (Y, {F^e}^*(\Syz_{j_0}(m))) =0$ for $e \gg 0$, then $R_m \subseteq I^F$.
\end{lemma}
\begin{proof}
Suppose that $H^{j_0}(Y, {F^e}^*(\Syz_j (m)))=0$ for $e \gg
0$. Choose $e$ big enough such that this cohomology group vanishes
and such that
$p^{e}(m-\alpha_{j,i})
\geq \twistell$ for all
$j,i$, where $\twistell$ is the number from Lemma 
\ref{cohomologyhopping-tight}.  Set $q=p^e$.
The first condition on $e$ implies that for any $f \in R_m$ and associated cohomology classes $c_j$, the class $c_{j_0}^q$ vanishes.  
The second condition implies that $H^j(Y, \O_Y(q(m-\alpha_{j,i})))=0$ for 
$j=1 \comdots \dip$ and $i=1 \comdots n_j$. 
With the same method as in the proof of Lemma~\ref{tools-tightclosure}, we conclude from the
cohomology sequences (\ref{cohomologysequencefromtwo}) and (\ref{cohomologysequencefromthree}), here with $\twistell'=0$,
that $f^q \in I^{[q]}$ for any $f \in R_m$. Hence $R_m \subseteq I^F$.
\end{proof}

\section{Semicontinuity argument}
\label{semicontinuity-section}

We want to show that an ideal inclusion in a polynomial ring $P$ in $\dip +1$
variables induces a tight closure inclusion for generic ideal generators in a standard-graded $k$-algebra $R$ of
dimension $\dip +1$. In this section we make precise what we mean by
generic and we describe a construction where we can apply the
semicontinuity theorem for projective morphisms. Note that $R$ is
not assumed to be normal in this section.

\newcommand{\varcoef}{{W}}
\newcommand{\numcoef}{{N}}
\newcommand{\dimpre}{{r}}
\newcommand{\subspaceprim}{{Z}}
\newcommand{\opennbh}{{U}}

\begin{definition} By a \emph{degree type} we shall mean an
$\nog$-tuple of natural numbers $(\degg_1 \comdots \degg_\nog)$. We say
that $f_1 \comdots f_\nog \in R$ are \emph{of degree type 
$(\degg_1 \comdots \degg_\nog)$} if
each $f_i$ is either zero or of degree $\degg_i$. 
\end{definition}

\begin{construction}
\label{construction}
Let $R$ be a standard-graded $k$-algebra of dimension $\dip +1$ and fix a degree type $(\degg_1, \ldots,
\degg_\nog)$, $\nog \geq \dip +1$. Let $R=k[y_1, \ldots,
y_\dimpre]/\mathfrak a$ be a homogeneous representation of $R$. For
a given degree
$\degg$ we may consider the element $G$ with indetermined
coefficients,
$$G= \sum_{\nu: \, \nu_1 \plusdots \nu_\dimpre = \degg} \varcoef_{\nu_1 \comdots
\nu_\dimpre} y_1^{\nu_1} \cdots y_\dimpre^{\nu_\dimpre} \, \in k[\varcoef_\nu] \otimes_k R \, .$$
For each $\degg_i$ we construct the corresponding indetermined
polynomial $G_i$ with new indeterminates. 
Let $\numcoef$ be the number of all indeterminates, and let $\mathbb A^\numcoef_k$ be the
corresponding affine space. As is well-known, the
$k$-rational points of $\mathbb A^\numcoef_k$ may be viewed as 
$N$-tuples of elements in $k$, so they parametrize 
all the possible ways to assign values in $k$ to the
indeterminates. Every such assignment yields concrete
polynomials $f_1 \comdots f_n \in R$ of degree type $(a_1 \comdots
a_{\nog})$. The condition ``the $f_i$ are all nonzero and generate an $R_+$-primary ideal'' defines an open, nonempty subset 
$\subspaceprim \subseteq \mathbb{A}^\numcoef_k$. An arbitrary not
necessarily closed point $z \in Z$ represents nonzero primary
elements $f_1 \komdots f_n \in \kappa(z) \otimes_k R$ of the given degree
type, and the generic point $\eta \in Z$ 
represents $G_1 \komdots G_n \in k(W_{i \nu}) \otimes_k R$.
\end{construction}

\begin{definition}
\label{genericdef}
We consider statements about objects which depend on choices of $f_1 \comdots f_\nog$ and hence on the parametrizing space $\subspaceprim$ (or $\mathbb
A^\numcoef_k$).

\begin{enumerate}
\item[(a)]
A statement will be said to hold for \emph{generic} ideal
generators if it holds for all rational points in a Zariski-open
non-empty subset of $\subspaceprim$.

\item[(b)]
A statement will be said to hold in the \emph{generic point} $\eta \in \subspaceprim$ if it holds
for the indetermined polynomials $G_1 \comdots G_\nog$ in
$k(\varcoef_{i \nu})
\otimes_k R$.

\item[(c)]
A statement will be said to hold
\emph{countably generically}
if it holds in the intersection of countably many non-empty open subsets of
$\subspaceprim$.
\end{enumerate}
\end{definition}

The statements we are most interested in are ``$R_m \subseteq
(f_1 \comdots f_\nog)^F $'' and ``$R_m \subseteq (f_1 \comdots f_\nog)^* $''.
As $k$ is an infinite field (we are assuming that $k$ is algebraically closed),
any non-empty Zariski-open subset of $\subspaceprim$ contains $k$-rational points. If $k$ is uncountably infinite, then the intersection of countably many open non-empty subsets of $\subspaceprim$ contains rational points (proof by induction on the dimension).
\newcommand{\pointsp}{{s}} 
\newcommand{\pointgen}{{z}} 

\begin{lemma}
\label{semicont}
Let $Y= \Proj R$ and consider the projection $X=Y \times_k
\subspaceprim
\rightarrow \subspaceprim$. Let $\firstsyzbig = \Syz_1(G_1 \comdots
G_\nog)$ be the first syzygy bundle on $X$ given by the indetermined
polynomials $G_i$ of the given degree type.
\begin{enumerate}
\item[(a)]
The sheaf $\firstsyzbig$ is locally free on $X$
and flat over $\subspaceprim$.

\item[(b)] Let $\pointgen \in \subspaceprim $ be a $k$-rational point with
corresponding ideal generators 
$f_1 \comdots f_\nog \in R$.
Then the fiber $(\firstsyzbig)_\pointgen$ over $Y=X_\pointgen$ is isomorphic
to 
$\Syz_1(f_1 \comdots f_\nog)$.

\item[(c)] If there exist
$f_1 \comdots f_\nog \in R$ such that
\[H^1(Y, {F^{e}}^* ( \Syz_1(f_1 \comdots f_\nog)(m)) \otimes \O(\twistell))=0\] ($e, \twistell, m$ fixed),
then this is true for generic choice of $f_1 \comdots f_\nog \in R$.
\end{enumerate}
\end{lemma}
\begin{proof}
The syzygy bundle $\firstsyzbig$ is given by
$$
\xymatrix{
0 \ar[r] & \firstsyzbig \ar[r] &
\bigoplus_{i=1}^\nog \O_X(-\degg_i) \ar[r]^<<<<<{G_1 \comdots G_\nog} &
\O_X \ar[r] & 0 \, ,}
$$
where $\O_X(1)$ is the relative very ample sheaf on $X$
corresponding to the standard grading of $k[\varcoef_{\nu,i}]
\otimes_k R$
(the $\varcoef_{\nu,i}$ have degree zero) and where the $G_i$ define a
surjection because of the primary condition. Hence $\firstsyzbig$ is
locally free and therefore by \cite[Proposition
III.9.2]{hartshornealgebraic}, flat over
$X$ and over $\subspaceprim$, which proves part (a).
Moreover, this short exact sequence restricts to $X_z \cong Y$ to the
defining sequence of
$\Syz_1(f_1 \comdots f_\nog)$, which gives part (b).
Part (c) follows from the semicontinuity theorem \cite[Theorem
III.12.8]{hartshornealgebraic} and from (a) and (b).
\end{proof}

If a polynomial ring $P \subseteq R$ is given over which $R$ is finite,
then the condition $f_1 \comdots f_\nog \in P$ defines a closed subset $V \subseteq \subspaceprim$.
Here is a picture to illustrate the situation.

\psfrag{Syz_z}{$(\firstsyzbig)_\pointgen$}
\psfrag{Syz_z'}{$(\firstsyzbig)_{\pointsp}$}
\psfrag{Syz_underline}{$\firstsyzbig$}
\psfrag{Y}{$Y$}
\psfrag{z}{$\pointgen$}
\psfrag{z'}{$\pointsp$}
\psfrag{V}{$V$}
\psfrag{Z}{$Z$}
\psfrag{Xdef}{$X=Y \times_k Z$}
\includegraphics[scale=0.8]{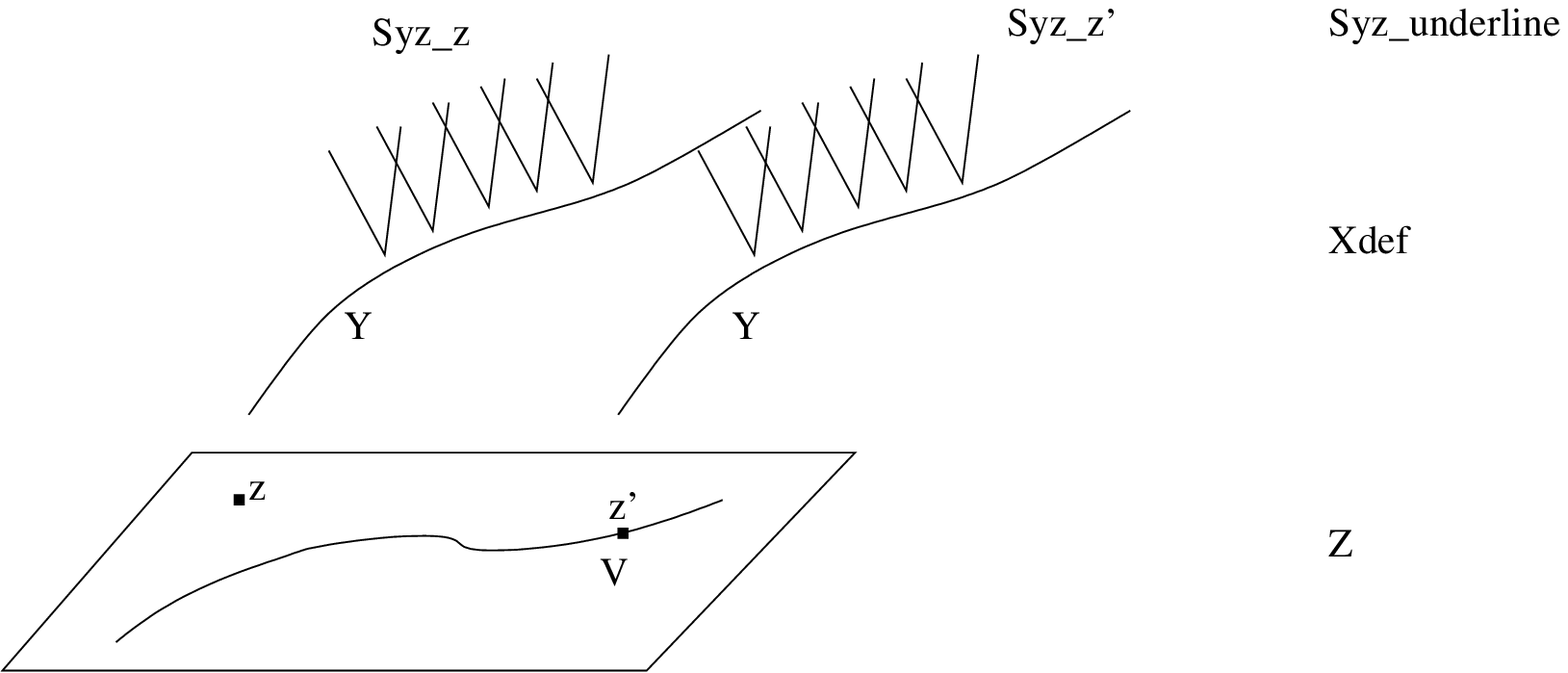}
\vskip2cm

In this picture, $(\firstsyzbig)_\pointsp$ and $(\firstsyzbig)_{\pointgen}$ are the
syzygy bundles on $Y$ associated to those ideal generators 
given by the points $\pointsp$ and $\pointgen \in \subspaceprim$ respectively. For points $\pointsp \in V$ we can use properties of the polynomial ring (e.g. that ideal resolutions are finite) to establish cohomological vanishing properties for the syzygy bundles.
These pass then over to an open neighborhood of $\pointsp$ inside
$\subspaceprim$.

\section{From ideal inclusion to generic bounds for tight closure and
Frobenius closure}
\label{ideal-to-tight}

Let $R$ be a standard-graded algebra over an algebraically closed  field $k$ of
characteristic $p>0$.
By the Noether normalization theorem \cite[Theorem 13.3]{eisenbud},
$R$ contains a polynomial ring
$P=k[x_0 \comdots x_{\dip}]$ such that $R$ is a finitely generated
$P$-module. For ideal generators $f_1 \comdots f_\nog \in R$ it is a
very special property to belong to $P \subseteq R$. In the following,
we will show that this property has strong consequences on the
containment in the ideal, its tight closure and its Frobenius closure
in $R$. These containments are characterized by cohomological
vanishing conditions which, by semicontinuity, pass over to generic
ideal generators not belonging to $P$.

\begin{remark} \label{topsyzsplitting}
Let $g_1 \comdots g_\nog \in P=k[x_0 \comdots x_{\dip}] \subseteq
R$ be homogeneous and $P_+$-primary. Consider the 
minimal free resolution of $P/(g_1 \comdots g_\nog)$
and the corresponding exact complex of sheaves (in the sense of (\ref{one})) on
$\Proj P = \PP^\dip$. This complex has length $\dip +1$, so it is of
the form
\begin{equation} \label{minrespolyring} 
0 \longrightarrow 
\bigoplus_{i=1}^{n_{\dip}}\O_{\PP^{\dip}}(-\alpha_{\dip,i}) \longrightarrow
\bigoplus_{i=1}^{n_{\dip-1}} \O_{\PP^{\dip}}(-\alpha_{\dip-1, i}) \longrightarrow \ldots
\longrightarrow \O_{\PP^{\dip}} \longrightarrow 0 \end{equation}
and the last syzygy bundle splits as
\begin{equation} \label{topsyzsplittingpolyring} \Syz_\dip =
\oplus_{i=1}^{n_\dip} \O_{\PP^\dip}(-\alpha_{\dip,i}). \end{equation}
Now consider the pullback of (\ref{minrespolyring}) 
via the finite morphism $Y=\Proj R \rightarrow \PP^{\dip}$,
\begin{equation} \label{minrespullback} 
0 \longrightarrow 
\bigoplus_{i=1}^{n_{\dip}}\O_Y(-\alpha_{\dip,i}) \longrightarrow
\bigoplus_{i=1}^{n_{\dip-1}} \O_Y(-\alpha_{\dip-1, i}) \longrightarrow \ldots
\longrightarrow \O_Y \longrightarrow 0 \end{equation}
This is an exact complex of sheaves on $Y$.
As on $\PP^{\dip}$, we have a splitting
\begin{equation} \label{topsyzsplittingR} 
\Syz_\dip = \oplus_{i=1}^{n_\dip} \O_{Y}(-\alpha_{\dip,i}).\end{equation}
\end{remark}

\begin{lemma} \label{alpha-estimate-lemma}
Consider the resolution (\ref{minrespolyring}) on $\PP^{\dip}$.
An inclusion $P_m \subseteq 
(g_1 \comdots g_\nog)$ gives an estimate for the Betti twists
$\alpha_{\dip,i}$, namely that
\begin{equation} \label{alpha-estimate} \alpha_{\dip,i} \leq m+\dip \, . \end{equation}
\end{lemma}

\begin{proof}
Assume this is false. Then $m' := \max \{ \alpha_{\dip,i}\}-\dip -1  \geq
m $ and the $m'$-twist of the resolution gives $H^\dip (\PP^\dip,
\Syz_\dip (m'))\neq 0$, because the bundle contains $\O_{\PP^\dip}(-\dip -1)$ as a direct
summand. On the other hand, because of
$\max \{ \alpha_{\dip,i}\} >  \max \{ \alpha_{\dip-1,i}\}  $ (by \cite[Exercise 20.19]{eisenbud}), we have
$H^\dip(\PP^\dip,\oplus_{i=1}^{n_{\dip -1}} \O_{\PP^\dip} (m' - \alpha_{\dip-1,i} )) = 0.$ Moreover, the intermediate cohomology of any splitting sheaf vanishes on $\PP^\dip$.
Therefore the long exact cohomology sequence associated to the sequence
\[ 0 \lra \bigoplus_{i=1}^{n_\dip} \O_{\PP^\dip}
(m'-\alpha_{\dip,i})=\Syz_\dip(m') \lra
\bigoplus_{i=1}^{n_{\dip-1}} \O_{\PP^\dip} (m'-\alpha_{\dip-1,i}) \lra
\Syz_{\dip -1}(m') \lra 0 \]
ends with
\[0 \lra H^{\dip-1}(\PP^\dip , \Syz_{\dip-1}(m')) \lra H^\dip(\PP^\dip, \Syz_\dip(m')) \lra 0. \]
Together with the isomorphisms (\ref{CM-coh-isomorphisms}), we obtain
\[0 \neq H^\dip(\PP^\dip, \Syz_\dip(m')) \cong H^{\dip -1}(\PP^\dip,\Syz_{\dip -1}(m') ) \cong \ldots \cong H^{1}(\PP^\dip,\Syz_{1}(m') )\] 
 and hence $P_{m'} \not\subseteq (g_1 \comdots g_\nog)$ -- a contradiction.
\end{proof}

\begin{lemma}
\label{step2}
Let $g_1 \comdots g_n \in P=k[x_0 \comdots x_{\dip}]$, $\dip \geq 1$, be homogeneous and $P_+$-primary,
and let $m \in \NN$.
Suppose that $P_m \subseteq (g_1 \comdots g_\nog)$ holds in $P$. Let $P \subseteq R$ be a graded Noetherian normalization. Then the following cohomological properties hold on $Y= \Proj R$.
\begin{enumerate}
\item[(a)] There exists $\twistell$ such that
\[H^1(Y,\Syz_1(g_1^q \komdots g_\nog^q) ( q(m+ \dip) + \twistell') ) =0\] holds for $q=p^e$, $e \geq 0$, and all $\twistell' \geq \twistell$.
\item[(b)] The cohomology module
\[H^1(Y,\Syz_1(g_1 \komdots g_\nog)(m+ \dip + 1)) \]
is annihilated by some Frobenius power.

\item[(c)] If $R$ is Cohen-Macaulay with $a$-invariant $\ainv$, then
\[H^1(Y,\Syz_1(g_1 \komdots g_\nog)(m+ \dip + 1 + \ainv)) = 0 \, .\]
\end{enumerate}
\end{lemma}
\begin{proof}
(a). Let $\tilde m := m+\dip$. We compare $\tilde m$ to the Betti twists $\alpha_{j,i}$ in the resolution (\ref{minrespolyring}): 
By Lemma~\ref{alpha-estimate-lemma},
we have 
\[\tilde m = m+\dip \geq \max_i \{ \alpha_{\dip,i} \} \geq \alpha_{j,i}\; \text{for
all}\; j \in \{1 \comdots \dip\}, i \in \{1 \comdots n_j\}.\]
Let $\twistell$ be as in Lemma~\ref{cohomologyhopping-tight}, i.e. such that $\twistell \geq 0$ and $H^j(Y, \O_Y(\twistell'))=0$ for $j=1 \comdots \dip$ and for all $\twistell' \geq \ell$. Then for
all $e \in \NN$ and all $\twistell' \geq \twistell$ we have

\begin{multline} \label{topcohvanishingtight} 
H^{\dip}(Y, F^{e*} (\Syz_{\dip}(\tilde m)) \otimes
\mathcal{O}(\twistell'))
= H^{\dip}(Y, \bigoplus_{i=1}^{n_{\dip}}
\O_Y(q(\tilde m-\alpha_{\dip,i})+\twistell') \\ 
=\bigoplus_{i=1}^{n_{\dip}} H^{\dip}(Y,\O_Y(q(\tilde m-\alpha_{\dip,i})+\twistell') =0 \; \text{(where}\; q=p^e \, \text{).} \end{multline}

Here the first equality follows from (\ref{topsyzsplittingR}). Since
$q=p^e>0$ and $\tilde m-\alpha_{\dip,i}\geq 0$, each twist appearing
in the direct sum expression is $\geq \twistell$, so the last equality
-- the vanishing of
the $\dip$-th cohomology -- follows from the assumption on
$\twistell$. By Lemma~\ref{tools-tightclosure}(c) (applied for $j=\dip$ and $\twistell_0=\twistell$), this implies part~(a) of the lemma.

(b). Now let $\hat m:=m+\dip+1$. Then $\hat m - \alpha_{j,i} >0$ for all $j,i$.
Let $\twistell$ be as above, and
let $e_0 \in \NN$ be large enough such that $p^{e_0} \geq \twistell$. 
Then for all $e \geq e_0$, we have
\begin{multline} \label{topcohvanishingFrob} 
H^{\dip}(Y, F^{e*} (\Syz_{\dip}(\hat m))) 
 = H^{\dip}(Y,\bigoplus_{i=1}^{n_{\dip}}
\O_Y(q(\hat m-\alpha_{\dip,i})) \\
=\bigoplus_{i=1}^{n_{\dip}} H^{\dip}(Y,\O_Y(q(\hat m-\alpha_{\dip,i})) =0 \, . \end{multline}

Since $q=p^e \geq p^{e_0}\geq \twistell$ and $\hat m-\alpha_{\dip,i}\geq 1$, each twist appearing
in the direct sum expression is $\geq\twistell$, so the last equality follows from the assumption on~$\twistell$. Since $H^j(Y, \O_Y(q( \hat m - \alpha_{j,i})))=0$ for all
$i,j$, the result follows from the cohomology sequences 
(\ref{cohomologysequencefromtwo}) and (\ref{cohomologysequencefromthree})  
with $\twistell'=0$.

(c). Recall that the $a$-invariant 
\cite[Section 3.6]{brunsherzog} of a graded ring $R$ is given by
$$\ainv
= \max \{\ell:\,H^{\dip+1}_{R_+}(R_\ell) \neq 0\}
= \max \{\ell:\, H^{\dip}(Y, \O_Y(\ell)) \neq 0 \} \, .$$ 
Now let $m' \geq m + \dip+1 + \ainv$. Then for $i \in \{ 1 \comdots n_{\dip} \}$, we have
\[ m' - \alpha_{\dip,i} \geq m+\dip+1 + \ainv - \alpha_{\dip,i} \geq \ainv +1\]
and hence $H^{\dip}(Y,\O_Y(m'-\alpha_{\dip,i}))=0$. It follows that
\begin{equation} \label{topcohvanishingCM} H^{\dip}(Y,\Syz_\dip(m')) = \bigoplus_{i=1}^{n_\dip} H^{\dip} 
(Y,\O_Y(m'- \alpha_{\dip,i})) =0. \end{equation}
Hence $c_\dip=0$ for the cohomology class corresponding to an element
$f \in R_{m'}$.
Because $R$ is Cohen-Macaulay, Lemma~\ref{CM-cohomologyhopping} gives the result and that $f \in (g_1 \comdots g_\nog)$.
\end{proof}

In Lemma \ref{step2} we can also deduce the containments $R_{m+\dip} \subseteq (g_1 \comdots g_\nog)^*$,
$R_{m+ \dip +1} \subseteq (g_1 \comdots g_\nog)^F $ (if $R$ is normal) and $R_{m + \dip +1+\ainv} \subseteq (g_1 \comdots g_\nog)$ (if $R$ is Cohen-Macaulay). We are however more interested in properties of the first cohomology module of the first syzygy bundle, because we want to apply semicontinuity to get these containments for general elements $(f_1 \komdots f_\nog)$, not only for the $(g_1 \comdots g_\nog)$ (we can not apply semicontinuity to the top-dimensional syzygy bundle, because this does not not dependend in  a flat way on the parameter space). We come to the main theorem of this paper. In short it says that an ideal inclusion for the polynomial ring yields a generic degree bound for tight closure and for Frobenius closure in every standard-graded algebra.

\renewcommand{\pointsp}{{s}}
\newcommand{\pointsec}{{z}}

\begin{theorem}
\label{maintheorem} Let $\dip \geq 1$. Let $\degg_1, \ldots, \degg_\nog$ be natural numbers, $\nog \geq \dip +1$. Let $m \in \NN$. Suppose that there exist $g_1 \comdots g_\nog \in P=k[x_0, \ldots, x_{\dip}]$, $\deg(g_i)=\degg_i$, such that $P_m \subseteq (g_1 \comdots g_\nog)$. Let $R$
be a $(\dip +1)$-dimensional standard-graded $k$-algebra with a graded Noether normalization $P \subseteq R$. Let $f_1 \comdots f_\nog \in R$ denote elements of degree type $(\degg_1 \comdots \degg_\nog)$.

\begin{enumerate}
\item[(a)] The containment
\[ R_{m+\dip} \subseteq (f_1 \comdots f_\nog)^* \]
holds in the generic point of $Z$ and countably generically 
(in the sense of Definition~\ref{genericdef}(b),(c)).

\item[(b)]
If $R$ is normal, then for generic elements $f_1 \comdots f_\nog \in R$ of the given degree type, we have
$R_{m+\dip +1} \subseteq (f_1 \comdots f_\nog)^F$.

\item[(c)]
If $R$ is Cohen-Macaulay with $a$-invariant $\ainv$ and of dimension $\geq 2$,
then
$R_{m+\dip +1 + \ainv } \subseteq (f_1 \comdots f_\nog)$
for generic elements of the given degree type.
\end{enumerate}
\end{theorem}
\begin{proof}
Consider the indetermined polynomials $G_1 \comdots G_\nog$, the parametrizing space $ \subspaceprim $ and the vector bundle $\firstsyzbig=\Syz_1(G_1 \comdots G_{\nog})$ on $X= Y \times_k \subspaceprim$ from Construction~\ref{construction}. Choose $g_1 \comdots g_\nog \in P$ such that $\deg(g_i)=\degg_i$ and $P_m \subseteq (g_1 \comdots g_\nog)$. This special choice of ideal generators is represented by a point $\pointsp$ in the parametrizing space $\subspaceprim $ such that $ (\firstsyzbig)_\pointsp =\Syz_1(g_1 \comdots g_{\nog})$.

(a). Let $\tilde m:=m+\dip$. By Lemma~\ref{step2}(a) we have
\[H^1(Y, F^{e*}(\Syz_1(g_1 \comdots g_\nog)(\tilde m)) \otimes \O_Y(\twistell')) =0\; \text{for}\; e \geq 0 \; \text{and} \; \twistell'\geq \twistell .\]
Fix $\twistell'\geq \twistell$ such that there exists $\test \in R_{\twistell'} \ra \Gamma(Y, \O_Y(\twistell'))$, $\test$ not in any minimal prime. Since $F^{e*}(\Syz_1(g_1 \comdots g_\nog)(\tilde{m})) \otimes \O_Y(\twistell') \cong (F^{e*}
(\firstsyzbig(\tilde{m})) \otimes \O_X(\twistell'))_\pointsp$,
we can apply Lemma \ref{semicont} and deduce that for fixed $e \geq 0$, this vanishing
holds for all points $\pointsec \in \opennbh_e$ in an open
neighborhood $ \opennbh_e \subseteq \subspaceprim$ of $\pointsp$.
Hence the vanishing condition for all $e \geq 0$ holds 
for all points in the countable
intersection $\bigcap_e \opennbh_e$. For the tuples $(f_1 \comdots f_\nog)$ represented by $k$-rational points in this intersection it follows from the proof of Lemma \ref{tools-tightclosure} that
$R_{\tilde{m}} \subseteq (f_1 \comdots f_\nog )^*$ in $\Gamma(D(R_+), \O)$.
Since this ring is inside the normalization of $R$, it follows that $R_{\tilde{m}} \subseteq (f_1 \comdots f_\nog)^*$ in the normalization. But tight closure can be computed in the
normalization anyway \cite[Theorem~1.7]{hunekeapplication}. Hence
$R_{\tilde{m}} \subseteq (f_1 \comdots f_\nog)^*$ holds countably generically
in the sense of Definition~\ref{genericdef}(c).

Since this intersection contains the generic point $\eta \in \subspaceprim$,
we infer again that the cohomological vanishing property holds and so again
by Lemma \ref{tools-tightclosure} (now over the field $k(\varcoef_{\nu,i})$) we get the containment.

(b). For the Frobenius containment let $\hat{m} > m + \dip +1$. By Lemma \ref{step2}(b) we know that the first cohomology of $\Syz_1(g_1 \comdots g_\nog)(\hat{m}) \cong (\firstsyzbig)_\pointsp(\hat{m})$ is annihilated by some Frobenius power. Hence by Lemma \ref{semicont} the same is true for $(\firstsyzbig)_\pointsec (\hat{m})$ for all points $\pointsec$ in an open neighborhood $\opennbh $ of
$\pointsp$. By Lemma \ref{tools-Frobeniusclosure} it follows for
these points that $ R_{\hat{m}} \subseteq (f_1 \comdots
f_\nog)^F$. Hence this containment holds for generic choice of
$f_1 \comdots f_\nog$.

(c). For the ideal containment we know by Lemma \ref{step2}(c) that \[H^1(Y, \Syz_1(g_1 \comdots g_\nog)(m'))=H^1(X_\pointsp, (\firstsyzbig)_\pointsp(m')) = 0 \] for $m' \geq m+\dip +1 + \ainv$.
By semicontinuity we get $H^1(X_\pointsec, \firstsyzbig_{\pointsec} (m'))  = 0$
for generic $\pointsec \in \subspaceprim$ and this means (by Lemma~\ref{CM-cohomologyhopping}) that $R_{ m'} \subseteq (f_1 \comdots f_\nog)$ for the ideal generators given by generic $\pointsec \in \subspaceprim$.
\end{proof}

Note that the countable intersection contains $k$-rational points, e.g. the point $\pointsp$ which represents $g_1 \komdots g_n$. Note also that in statement (c) for ideal inclusion, the bound depends on an invariant of the ring, whereas the bounds for tight closure and Frobenius closure does not.

\newcommand{\boundk}{{\gamma}}
\newcommand{\coelin}{{\beta}}

\begin{remark}[Analogues for zero- and one-dimensional rings]
\begin{enumerate}
\item[(a)] Let $k$ be a field of positive characteristic, and let
$R$ be a standard-graded $k$-algebra of dimension zero. Then we have
$I^*=I^F=R_+$ for every ideal $I
\subseteq R_+$.
\item[(b)] Let $R$ be a standard-graded algebra of dimension
one over any field~$k$, and let $f$ be a generic element of degree
$\degg$ of $R$. Then there exists $\boundk \in \NN$ such that for all $\coelin
\in \NN$, we have $R_{\coelin \degg+\boundk} \subseteq (f^\coelin)$. 
\item[(c)] Let $R$ be a standard-graded algebra of
dimension one over a field~$k$ of positive characteristic.
Let $f$ be a generic element of $R$ of degree~$\degg$, and 
let $I:=(f) \subseteq R$.
Then $R_{\degg+1} \subseteq I^F$ and $R_{\degg} \subseteq I^*$.
\end{enumerate}
\end{remark}
\begin{proof}
(a) By definition, every prime ideal in a zero-dimensional ring $R$ is
maximal.  But the only maximal ideal of $R$ is the irrelevant ideal
$R_+$, since $R_+$ is the only homogeneous maximal ideal and every
maximal ideal $\mathfrak{m}$ 
is equal to the (prime) ideal generated by the homogeneous 
elements of $\mathfrak{m}$. Hence every $f \in R_+$ is nilpotent and
we obtain the desired result.

(b) Write $R$ as a quotient of a polynomial ring $k[x_0
\comdots x_r]$. From an elementary
calculation starting with integral equations for the $x_i$ over the
subring $k[f]$, one obtains the desired result.

(c) Let $g \in R_{\degg+1}$. Choose $e \in \NN$ such that $q=p^e \geq\boundk$ with the
$\boundk$ from part~(b). Then $g^q$
has degree $q(\degg+1)=q\degg+q \geq q \degg+\boundk$. By part~(b), we have
$g \in (f^q) = I^{[q]}$ and hence $g \in I^F$.

Now let $g \in R_{\degg}$. Let $\test \in R$ be of degree $\geq \boundk$, $\test
\notin$ minimal prime. Then
for $q=p^e> 0$, $\deg (\test g^q) \geq \boundk +q \degg$. By part~(b),
this implies $\test g^q \in I^{[q]}$ for all $q > 0$ and thus $g \in
I^*$.
\end{proof}

\begin{remark}
\label{parameter}
Suppose that $\nog= \dip +1$. This is the parameter case (a generic
choice always yields parameters),
and the Koszul resolution yields $\Syz_\dip = \O(- \sum_{i=1}^\nog
\degg_i)$. Hence on the polynomial ring we have $P_{ \geq m}
\subseteq  (g_1 \comdots g_\nog)$ with $m= \sum_{i=1}^\nog \degg_i -
\dip$. Theorem \ref{maintheorem} yields $R_{\sum_{i=1}^\nog
\degg_i} \subseteq (f_1 \comdots f_\dip)^*$ for arbitrary parameters
in $R$, which is well known \cite[Theorem 2.9]{hunekeparameter},
and $R_{\sum_{i=1}^\nog \degg_i +1} \subseteq (f_1 \comdots
f_\dip)^F$.

Already the parameter case shows that the bounds of Theorem
\ref{maintheorem} can not be improved
(as a bound holding for all $R$ of the given dimension).
The vanishing theorem of
Hara (\cite[Theorem~6.1]{hunekeparameter}; 
\cite{hararationalfrobenius}) shows that an element $f \in R_{\tilde{m}}$, $\tilde{m} <
\sum_{i=1}^\nog \degg_i $, does not belong to $I^*$ unless it belongs to $I$
itself. But $f \in I$ does not hold in general.
\end{remark}

\begin{remark}
\label{plus}
The \emph{plus closure} of an ideal $I \subseteq R$, $R$ a domain, is
defined by $I^+ = \{f \in R:\, \text{there exists } R \subseteq S
\text{ finite}, f \in IS\}$. There are inclusions $I^F \subseteq I^+
\subseteq I^*$. It was only proved recently that $I^+ \neq I^*$
(\cite{brennermonsky}). Does a generic inclusion $R_{m+\dip}
\subseteq I^+$ hold, with $m$ as in Theorem~\ref{maintheorem}? 
For $R_{\geq m+ \dip +1}$ the answer is yes via the Frobenius
closure. An analogue of Lemma \ref{step2}(a) does also hold for the
(graded) plus closure: The cohomology groups $H^\dip(Y, \O_Y(\tilde m
- \alpha_{\dip,i}))$ (and hence $H^{\dip}(Y, \Syz_{\dip}(\tilde m))$ and $H^{1}(Y, \Syz_{1}(\tilde m))$, where
$\tilde m = m+\dip$) can
be annihilated by a finite cover $Y' \rightarrow Y$ due to the graded
version of \cite[Theorem~5.1]{smithparameter}. 
However, this does not mean that $H^1(Y',
\Syz_{1}(\tilde m))=0$, so there is no basis for a semicontinuity argument.
\end{remark}

\begin{remark}
Since $I^F \subseteq I^*$, Theorem \ref{maintheorem}(b) remains true if we
replace $I^F$ by $I^*$, so it also yields a \emph{generic} bound (in the
sense of Definition~\ref{genericdef}(a)) for
tight closure in the normal case. Note that unlike tight closure,
Frobenius closure does not commute with normalization. For example,
let
$R=k[x,y]/(xy)$ and let $I:=(ax+by)\subseteq R$ with $a,b \neq
0$. Then we have
$x \in R \cap IR^{\rm{norm}} \subseteq R \cap (I R^{\rm{norm}})^F$, but $x \notin I^F$.
\end{remark}

\begin{remark}
\label{charzero}
There are several problems to get similiar statements in characteristic zero. One problem is how the open subsets of the parameter space on which the generic inclusion statements hold vary with the prime characteristic and whether there exists an open subset in the relative setting. This is also a problem for the degree bound which ensures containment in the Frobenius closure.

In working with the generic point in characteristic zero, that is, over the field $\QQ(W_{i \nu})$, we have another problem: we can consider the base ring $\ZZ(W_{i \nu})= \ZZ[W_{i \nu}]_S$, where $S$ is the multiplicative system of all primitive polynomials, and the residue class fields are $\QQ(W_{i \nu})$ and $\ZZ/(p)(W_{i \nu})$. We have tight closure inclusions over all closed points (of positive characteristics), but this base ring is not of finite type over $\ZZ$, so it can not be taken as an arithmetic basis to get tight closure in characteristic zero.
\end{remark}

\section{Generic ideal inclusion and the Fr\"oberg conjecture}
\label{ideal-inclusion}

By the results of the previous sections, we know that if $k$ is a field of positive characteristic, then a degree bound for ideal inclusion in the 
polynomial ring $P=k[x_0 \comdots x_\dip]$
yields a generic degree bound for Frobenius closure and for tight
closure in any standard-graded $k$-algebra of dimension $\dip+1$. 
Therefore we have to look for good inclusion bounds for $P$
for a given degree type. So let 
$\degg_1 \comdots \degg_\nog$ be natural numbers, and
let $g_1 \comdots g_\nog \in P$ be homogeneous polynomials of 
degree $\deg(g_i)=\degg_i$. 
We seek to describe the smallest number $m_0 = \min \{m:\, P_m \subseteq (g_1 \comdots g_\nog)\}$ in terms of $\nog$
and $\degg$, and for generic choice of the $g_i$.

For fixed $(g_1 \comdots g_\nog)=I$ the number $m_0$ is the smallest 
zero of the
\emph{Hilbert function}
\[ H(m):= \dim _k (P/I)_m \, .\]

This can be estimated using a piecewise polynomial function as
follows.

\newcommand{\Degg}{A}
\newcommand{\Subdegg}{B}
\newcommand{\subdegg}{b}
\newcommand{\subnog}{r}

\begin{definition}
\begin{enumerate}
\item[(a)]
For $\Degg=(\degg_1 \comdots \degg_\nog) \in \NN^\nog$ and
$\Subdegg=(\subdegg_1 \comdots \subdegg_\subnog) \in \NN^\subnog$, we
write $\Subdegg \subseteq \Degg$ if there exists a permutation
$\Subdegg'$ of $\Subdegg$ and a subset $\{i_1 \comdots i_\subnog\}
\subseteq \{1 \comdots \nog\}$ such that $\Subdegg' = (\degg_{i_1}
\comdots \degg_{i_{\subnog}})$. Moreover, we put
$\ell(\Subdegg):=\subnog$ and $|\Subdegg|:=\subdegg_1+ \ldots + \subdegg_\subnog$.
\item[(b)]
For given $\nog, \dip \in \NN$, 
$\Degg=(\degg_1 \comdots \degg_\nog) \in \NN^\nog$, and $m \in \NN_0$,
set
\[F(m):=F(\nog, \dip, A, m):=\sum_{\Subdegg \subseteq \Degg} (-1)^{\ell(B)}
\binom{\dip +m -|\Subdegg|}{\dip}\]
where $\binom{N}{K}$ is taken to be zero unless $N \geq K \geq 0$. 
In particular, for constant degrees $\degg_1 = \ldots = \degg_\nog=\degg$,
we have
\[ F(m) = \sum_{s=0}^\nog (-1)^s \binom{\nog}{s} \binom{ \dip-2+m-s
\degg}{\dip} \, . \]
Moreover, we set
\[F^+(m):=\max\{0, F(m)\}.\]
\end{enumerate}
\end{definition}

We know \cite{anick, valla} that
\[ H(m) \geq F^+(m)\]
for all $m \in \NN_0$. For constant degrees 
$\degg_1= \ldots =\degg_\nog=:\degg$, this is
equivalent to the coefficient-wise inequality of formal power series
\[ \mathcal{H} (\lambda) \geq \left|\frac{(1 -
\lambda^{\degg})^{\nog}}{(1-
\lambda)^{\dip +1}} \right|,\]
where the absolute value symbols denote the initial non-negative
segment of the power series (obtained by replacing all negative
coefficients with zero) and where $\mathcal{H} (\lambda) $ denotes
the \emph{Hilbert series} whose coefficients are $H(m)$.

\begin{conjecture}[Fr\"oberg \cite{froeberg}]  
For generic ideal generators $g_1 \comdots g_\nog$, we have 
\[ H(m) = F^+(m)\]
for all $m \in \NN_0$.
\end{conjecture}

\begin{proposition}
\cite[Section 4]{chandlergeomfi}
The Fr\"oberg conjecture holds in the following cases:
\begin{itemize}
\item $\nog \leq \dip +1$, and $\nog = \dip+2$ in characteristic $0$:
Iarrobino \cite{iarrobinocompressedalg}.
(In the case $\nog \leq \dip +1$, $\nog$ generic ideal generators always form a regular sequence and
this implies that the conjecture holds, see
\cite[Sections~1 and~4]{valla}.
\item $\nog = \dip+2$: Stanley \cite{stanleyhilbert}
\item $\dip = 1$: Fr\"oberg \cite{froeberg}
\item $\dip =2$: Anick \cite{anick}
\end{itemize}
\end{proposition}

Suppose that we are in a situation (as listed above) where the Fr\"oberg conjecture holds. Then the smallest
positive zero $m_0$ of $H(m)$  is the smallest zero of $F^+(m)$.
Thus finding this zero amounts to solving polynomial equations of degree $\dip$.
This can be done in several cases and then yields explicit generic ideal
inclusion bounds for the polynomial ring, which yield generic tight
closure bounds via Theorem~\ref{maintheorem}. 

\begin{notation}
In the following, we assume that all ideal generators have the same
degree, which we denote $\degg$.
\end{notation}

In the parameter case, i.e. when $\nog = \dip +1$, then 
\[m_0 = \nog \degg - \dip \]
is the generic ideal inclusion bound for $\nog$ elements of degree
$\degg$ in the polynomial ring of dimension $\dip+1$.

In the ``almost-parameter case'', i.e. when $\nog=\dip+2$, we have the
following result by Migliore and Mir\'o-Roig.
\begin{proposition}
\cite[Lemma~2.5]{miglioremiroroigminimal}
\label{mimi}
Let $\nog =\dip+2$. Then
\[m_0:=\bigm\lfloor \frac{1}{2}( \nog \degg-\nog) \bigm\rfloor +1\]
is the generic ideal inclusion bound for $\nog$ elements of
degree~$\degg$ in the
polynomial ring of dimension $\dip +1$.
\end{proposition}

\begin{corollary}
Let $R$ be a standard-graded $k$-algebra of dimension $\dip +1$. Then
\[ R_{\geq \lfloor \frac{1}{2} (\dip+2)(\degg+1) \rfloor  -1} \subseteq I^* \]
holds for an ideal $I$ generated by $\dip+2$ generically chosen
elements of degree~$\degg$.
\end{corollary}
\begin{proof}
This follows from Proposition \ref{mimi} and Theorem
\ref{maintheorem} by a direct calculation.
\end{proof}

We now look at the first positive zero of $F^+(m)$ in the low-dimensional
cases where the Fr\"oberg conjecture holds, i.e. in the cases 
$\dip=1$ and $\dip=2$. 

\begin{lemma}
\label{dim2bound}
Let $\dip=1$. For $\nog \geq 2$ and constant degree $\degg$, the number
$$m_0=\bigm\lceil\frac{\nog}{\nog-1} \degg \bigm\rceil -1$$
is the smallest zero of $F^+(m)$.
\end{lemma}
\begin{proof}
For $m < 2 \degg -1 $ we have to consider in the formula for $F(m)$ only the first two summands,
and we find in this range the zero at
$m =  \frac{\nog}{\nog-1} \degg -1$.
\end{proof}

\begin{lemma}
\label{dim3bound}
Let $\dip=2$. For constant degree $\degg$, the smallest zero of $F^+(m)$ is
\[m_0=3\degg-2\; \text{if}\; \nog=3\] and
\[ m_0= \Bigm \lceil \frac{1}{2(\nog-1)} \bigm(
3-3\nog+2\degg \nog + \sqrt{1-2\nog + \nog^2 + 4 \degg^2 \nog}
\bigm)
\Bigm
\rceil 
\; \text{if}\; \nog \geq 4.\]
\end{lemma}
\begin{proof}
We extend $F(m)$ to a function $\tilde F(r)$ of real numbers by setting
$ \binom{ r}{ k} := \frac{r (r-1)
\ldots (r-k+1)}{k!}$ for $r \in \RR, r \geq k \geq 0$,
$ \binom{ r}{ k} :=0$ otherwise, and
\[ \tilde F(r):= \sum_{s=1}^{\nog} (-1)^s \binom{ \nog }{ s } \binom{ 2 + r - s \degg}{ 2} \,  .   \]
Thus $\tilde F(r)$ is a function from $\RR$ to $\RR$
that agrees with the ``Fr\"oberg function'' $F(m)$ for $r=m \in \NN$
and is continuous everywhere except at the places $r=s \degg$, for
$s=1 \comdots \nog$.

On any interval where $\tilde F(r)$ is continuous, it agrees with a quadratic
polynomial in $r$, and if this polynomial has a zero in the relevant
interval then this gives a zero of $F^+(m)$. For $\nog \geq 4$, we
find in the interval $[\degg, 2 \degg -1]$ a zero as given above.
For $\nog=3$, there are no zeros in this interval and we find the
first zero in the interval $[2 \degg, 3 \degg-1]$.
\end{proof}

\begin{remark}
Note that our last result is compatible with the result by
Migliore and Mir\'o-Roig. For $\dip=2$, $\nog=4$ and constant degree
$\degg$, their result is
$2\degg-1$. One easily checks that
\[\frac{1}{6}\bigm(3-12+8\degg+\sqrt{9+16\degg^2}\bigm) - (2\degg-1)
\in (-1, 0)\]
and thus the zero of $F^+(m)$ given by Lemma \ref{dim3bound} is
equal to $2\degg-1$.
\end{remark}

\section{A slightly better bound for strongly semistable syzygy sheaves}
\label{semistab}

We describe another situation where the cohomological conditions on
$\Syz_\dip$ are also fulfilled and where again the inclusion into
tight closure resp. Frobenius closure can be deduced.
Let $\mathcal{S}$ be a locally free sheaf on $Y= \Proj R$. Then the
\emph{slope} of $\mathcal{S}$ is defined to be
$$\mu(\mathcal{S}):= \frac{\deg(\mathcal{S})}{\rank(\mathcal{S})}$$
(for the degree fix an ample invertible sheaf). The sheaf $\mathcal{S}$ is called \emph{semistable} if
$\mu(\mathcal{T}) \leq \mu(\mathcal{S})$ holds for all coherent subsheaves
$\mathcal{T} \subseteq \mathcal{S}$.
A locally free sheaf is called
\emph{strongly semistable} if all of its Frobenius
pullbacks are semistable. A semistable locally free sheaf
$\mathcal{S}$ of negative degree has no global sections.

\begin{proposition}
\label{semistablecoho}
Let $\mathcal S$ be a strongly semistable locally free sheaf on $Y=\Proj R$ of dimension $\dip$,
where $R$ is a normal standard-graded $k$-domain and where $Y$ is smooth.
\begin{enumerate}
\item[(a)]
If $\deg (\shS) \geq 0$, then there exists $\twistell_0$ such that
$H^\dip(Y, {F^{e}}^{*}(\shS) \otimes \O(\twistell'))=0$ for all $e$ and all $\twistell' \geq \twistell_0$.

\item[(b)]
If $\deg (\shS) > 0$, then
$H^\dip(Y, {F^{e}}^{*}(\shS))=0$ for all $e \gg 0$.
\end{enumerate}
\end{proposition}
\begin{proof}
By Serre duality we have 
$$H^\dip(Y, {F^{e}}^{*}(\shS) \otimes
\O(\twistell')) \cong H^0(Y, {F^{e}}^{*}(\shS^\dual)  \otimes \O_Y(-\twistell') \otimes \omega_Y).$$
For (a) choose $\twistell_0$ such that
$\O_Y(-\twistell_0) \otimes \omega_Y $ has negative degree. Then the
whole sheaf has negative degree and has, because it is semistable,
no global sections. For (b) we have $H^\dip(Y, {F^{e}}^*(\shS) )
\cong H^0(Y, {F^{e}}^*(\shS^\dual)
\otimes \omega_Y)$.
But for $e
\gg 0$ the sheaf
${F^{e}}^{*}(\shS^\dual) \otimes \omega_Y $ has again negative degree
and no sections.
\end{proof}

With this we can slightly improve the generic degree bound for Frobenius
closure under special conditions.

\begin{corollary}
\label{semistablecor}
Suppose that $\Syz_{\dip}(m)$ is strongly semistable of
positive degree. Then we have $R_m \subseteq I^F$.
\end{corollary}
\begin{proof}
This follows by applying Proposition \ref{semistablecoho}(b) to
$\shS=\Syz_\dip(m)$ and then using Lemma \ref{tools-Frobeniusclosure}.
\end{proof}

When does this give a better bound?
Let $r_0$ denote the smallest zero of the
continuation $\tilde F(r)$ of the Fr\"oberg function on
$\RR_{>0}$, and suppose that this is not a natural number, so that $m_0 > r_0$.
Then $\Syz_\dip(m_0+\dip)$ has positive degree, but it contains also
$\O_{\PP^\dip}$ as a direct summand (and is therefore not semistable). In this situation it may happen
that for $R \supseteq P$ the generic last syzygy bundle has also
positive degree and is moreover strongly semistable. This behavior
occurs in the two-dimensional situation, as the following example
shows, but it is difficult to establish in higher dimensions.

\begin{example} \label{tabellen}
Let $R$ be normal and two-dimensional, so that $Y=\Proj R$ is a
smooth projective curve. Consider the constant degree type
$(\degg,\degg,\degg) $ with $\degg$ odd.
Then for generic choice on $\PP^1_k$ we have
$\Syz_1 = \O(-\frac{3 \degg +1}{2})
\oplus \O(-\frac{3\degg -1}{2})$ and the best inclusion is $P _{\geq \frac{3\degg -1 }{2}} \subseteq (g_1,g_2,g_3)$. Hence
Theorem \ref{maintheorem} yields $R_{\geq\frac{3\degg+3 }{2}}
\subseteq (f_1,f_2,f_3)^F$ for generic choice of $(f_1,f_2,f_3)$.
However, under the condition that $\Syz(f_1,f_2,f_3)$ is strongly 
semistable, Corollary \ref{semistablecor} yields the better inclusion
$R_{\geq\frac{3\degg+1 }{2}} \subseteq (f_1,f_2,f_3)^F$.
\end{example}

\section{Examples and asymptotic behavior of tight closure bounds}
\label{examples}

We are going to compare our generic tight closure bound from Theorem \ref{maintheorem} with previously known bounds, which depend on several different conditions.

\begin{proposition}[Koszul bound] (cf. \cite[Proposition 3.3]{smithgraded}, \cite[Corollary 2.6]{brennerlinearbound})
Let $f_1 \comdots f_\nog \in R$ be elements of degrees $\degg_1 \geq
\degg_2 \geq \ldots \geq \degg_\nog$ that
define an $R_+$-primary ideal. Then
\[ m_0=\sum_{i=1}^{\dip+1} \degg_i \]
is an inclusion bound for the tight closure of $(f_1 \comdots f_\nog)$.
In particular, for constant degree $\degg$, the bound is
\[ m_0=(\dip+1) \degg. \]
\end{proposition}

\begin{proof} Consider the Koszul complex
\cite[Chapter 17]{eisenbud}, which is exact on the punctured spectrum. The surjection
$$\bigoplus_{I \subseteq \{1 \comdots \nog \},\, |I|=\dip +1  }\O_Y(m-\sum_{i \in I }\degg_i )
\lra \Syz_\dip(m) \lra 0$$ on
$Y = \Proj R$ yields a surjection for the $\dip$-th cohomology.
For $m \geq   \sum_{i=1}^{\dip +1}  \degg_i$, all twists in the
summands on the left are non-negative.
Now let $\twistell$ be as in Lemma~\ref{tools-tightclosure}. 
Then for
$\twistell' \geq \twistell$, we obtain a surjection
$$H^\dip\big(\bigoplus_{|I|=\dip +1}
\O_Y(q(m-\sum_{i \in I }\degg_i)+ \twistell' )\big)
\lra H^\dip\big(F^{e*}(\Syz_\dip(m)) \otimes \O_Y(\twistell')\big) \lra 0.$$
The term on the left is zero due to the assumption on $\twistell$, so
we obtain that $H^\dip\big(F^{e*}(\Syz_\dip(m)) \otimes
\O_Y(\twistell')\big)=0$. Hence Lemma~\ref{tools-tightclosure}(c)
yields that for primary ideal generators, the maximum
over all degree sums of length $\dip +1$ is a tight closure bound.
\end{proof}

The Koszul complex was used in \cite{brennerlinearbound} to establish a better bound. Under the condition that $\Syz_1$ is strongly semistable, one obtains that
$\Syz_\dip$ is (as, up to a twist, an exterior product of $\Syz_1$)
also strongly semistable, and this gives the following bound.

\begin{proposition}[Bound when $\Syz_1$ is strongly semistable]
Let $f_1 \comdots f_\nog \in R$ be elements of degrees $\degg_1 \geq
\degg_2 \geq \ldots \geq \degg_\nog$ that
define an $R_+$-primary ideal. Suppose that $\Syz_1(f_1 \comdots
f_\nog)$ 
is strongly semistable.  Then
\[ m_0= \lceil \frac{\dip}{\nog-1} (\degg_1 + \ldots + \degg_\nog)
\rceil \] 
is an inclusion bound for the tight closure of $(f_1 \comdots
f_\nog)$. In particular, for constant degree $\degg$, the bound is 
\[ m_0=\lceil \frac{\dip \nog}{\nog-1} \cdot \degg \rceil. \]
\end{proposition}

In ring dimension two (i.e. when $\dip =1$), this bound coincides with the generic
bound from Theorem \ref{maintheorem}, but in higher dimensions the
new bound coming from generic ideal inclusion in the polynomial ring
is much better (mainly because the Koszul complex is not minimal in
general).

\begin{example}
Let $f_1 \comdots f_\nog$ have degree $\degg=10$. We state the bounds
for the containment in $I^*$ depending on $\dip$ and $\nog$. Of
course, the conditions for the bounds to hold differ for each bound.
In the last column of each table, we state the limit of the bounds
for
$\nog \ra \infty$. Let $m_0$ be the generic ideal inclusion bound in the
polynomial ring.

\medskip

$\dip=1$:

\smallskip

\begin{tabular}{l|l|llllllll|l}
$\nog$ & &2 & 3 & 4 & 5 &6&7&10& 11 & $\infty$ \\
\hline
Koszul bound & $m=2 \cdot \degg$ & 20 & 20 &  20 &  20 & 20 & 20 & 20 &20 & 20\\
$\Syz_1$ str. semist.& $m = \lceil \frac{\nog}{\nog-1} \degg \rceil$
& 20 & 15 & 14 & 13 &12   &12 &12 & 11 & 11 \\
Generic bound& $m=m_0+1$
& 20 & 15 & 14 & 13&12 &12  & 12 &  11 & 11
\end{tabular}

\medskip

$\dip=2$:

\smallskip
\begin{tabular}{l|l|llllllll|l}
$\nog$ & &3 & 4 & 5 & 6 & 7& 8 & 10 & 11 &$\infty$ \\
\hline
Koszul bound & $m=3 \cdot \degg$ & 30 & 30 & 30 & 30 & 30 & 30 & 30 & 30 & 30 \\
$\Syz_1$ str. semist. & $m = \lceil \frac{2\nog}{\nog-1} \degg \rceil$
& 30 & 27 & 25 & 24 & 24& 23 &23 &22& 21 \\
Generic bound & $m=m_0+2$ & 30 & 21 & 19  & 18 & 17 & 16 & 16 & 15 &12
\end{tabular}

\medskip

$\dip=3$:

\smallskip
\begin{tabular}{l|l|llllllll|l}
 & $\nog$ &4 & 5 & 6 & 7 & 8 & 9 & 10 & 11 & $\infty$ \\
\hline
Koszul bound & $m=4 \cdot \degg$ & 40 & 40 & 40 & 40 & 40 & 40 & 40 &
40 & 40 \\
$\Syz_1$ str. semist.& $m = \lceil \frac{3\nog}{\nog-1} \degg \rceil$
& 40 & 38 & 36 & 35 & 35 & 34 & 34 & 33 & 31 \\
Generic bound & $m=m_0+3$ & 40 & 26 & 24  & 22 & 22 & 21 & 20 & 20 & 13 
\end{tabular}
\end{example}

\begin{remark}
Let $\dip=3$. If $\nog \geq 6$, then
we do not have a closed formula for the first zero of the Fr\"oberg
function and we do not know whether the Fr\"oberg conjecture holds. 
However, we can do a computational check of the Fr\"oberg conjecture
and compute the first zero of $F^+(m)$ \emph{for specific
examples}. For the above example, we computed the Hilbert
function of $\nog$ randomized ideal generators, then the Fr\"oberg
function for $n$ generators of degree $10$, checked that the values agree
and in what degree they first become non-positive. The search for a zero of the Fr\"oberg function does not really
require the functionality of \emph{CoCoA} and has independently been
realized in \emph{Lisp}.
\end{remark}

\begin{remark} 
[Asymptotic behavior for $\nog \ra \infty$]
If the degree $\degg$ is fixed and $\nog$ grows as large as the
number of monomials in $\dip +1$ variables of degree $\degg$, then the
Fr\"oberg conjecture holds for trivial reasons with $m_0 =\degg $.
Therefore the asymptotic limit of our generic bound is $\degg + \dip$
(the true limit, even for ideal inclusion in $R$, is of course also $m_0$,
but this is obtained for much larger $\nog$ which also depends on $R$).
\end{remark}

\begin{remark} \label{asymptotic-large-degree}
[Asymptotic behavior for $\degg \ra \infty$]
Let $\nog \geq \dip +1$ be fixed, and let $\degg \ra \infty$.
The Koszul bound is just $(\dip+1) \cdot \degg$. The bound which holds
for strongly semistable first syzygy bundles behaves asymptotically
like
$\frac{\dip \nog}{\nog-1} \cdot \degg$.

Our new bound behaves for $\dip=2$ like
$\frac{\nog+\sqrt{\nog}}{\nog-1} \cdot \degg$, so it is considerably lower in
the limit (the coefficent can also be written as
$\frac{\sqrt{\nog}}{\sqrt{\nog} -1}$).

For $\dip=3$ we do not have a closed formula for the
bound coming from the Fr\"oberg function, but
\emph{Lisp} computations suggest that if $\nog$ is a cube,
i.e. if $\nog=\twistell^3$ for some $\twistell \in \NN$, then the
smallest positive zero of
$F^+(m)$ behaves like $\frac{\twistell}{\twistell-1} \cdot \degg$ for $\degg \gg 0$.
So assuming that the Fr\"oberg conjecture holds, we are led to
conjecture that the generic (ideal and) tight closure bound behaves
like $\frac{\twistell}{\twistell-1}\cdot \degg$ for
$\nog=\twistell^3$. However, we have so far been unable to find an
algebraic expression when $\nog$ is not a cube.

We have also computed the zeros of the
Fr\"oberg function for $\dip=4$ and $\nog=\twistell^4$ with $\twistell \in \NN$. For
small $\twistell$, we obtained again that it behaves like
$\frac{\twistell}{\twistell-1}\cdot \degg$.
\end{remark}

\bibliographystyle{amsplain}

\bibliography{bibliothek}

\end{document}